\documentclass[sn-mathphys,Numbered]{sn-jnl}
\usepackage[utf8]{inputenc}
\usepackage{amsmath}
\usepackage{amsthm}
\usepackage{amsfonts}
\usepackage{amssymb}
\usepackage{graphicx}
\usepackage{bbold}
\usepackage{gensymb,xcolor}
\usepackage{graphicx}%
\usepackage{multirow}%
\usepackage{mathrsfs}%
\usepackage[title]{appendix}%
\usepackage{xcolor}%
\usepackage{textcomp}%
\usepackage{manyfoot}%
\usepackage{booktabs}%
\usepackage{algorithm}%
\usepackage{algorithmicx}%
\usepackage{algpseudocode}%
\usepackage{listings}%

\newcommand{\E}{\mathbb E}
\newcommand{\Prob}[1]{\mathbb{P} \left\{#1\right\}}
\newcommand{\R}{\mathbb{R}}
\newcommand{\N}{\mathbb{N}}
\newcommand{\one}{\boldsymbol{1}}
\newcommand{\eps}{\varepsilon}
\newcommand{\Sphere}[1][d-1]{\mathbb{S}^{#1}}
\newcommand{\sH}{\mathcal{H}}

\newcommand{\sP}{\mathcal{P}}
\newcommand{\sF}{\mathcal{F}}

\newcommand{\sR}{\mathcal{R}}
\newcommand{\sU}{\mathcal{U}}
\newcommand{\dto}{\xrightarrow{\text{d}}}

\DeclareMathOperator{\exo}{exo}

\DeclareMathOperator{\supp}{supp}
\DeclareMathOperator{\esssup}{ess\,sup}
\DeclareMathOperator{\interior}{Int}
\DeclareMathOperator{\cl}{cl}

\theoremstyle{plain}
\newtheorem{theorem}{Theorem}[section]
\newtheorem{proposition}[theorem]{Proposition}

\newtheorem{lemma}[theorem]{Lemma}
\newtheorem{corollary}[theorem]{Corollary}

\theoremstyle{definition}
\newtheorem{definition}[theorem]{Definition}

\theoremstyle{remark}
\newtheorem{remark}[theorem]{Remark}
\newtheorem{example}[theorem]{Example}

\begin{document}

\title{Intersections of randomly translated sets}
\author{\fnm{Tommaso} \sur{Visonà}}\email{tommaso.visona@unibe.ch}
\affil{\orgdiv{Institute of Mathematical Statistics and Actuarial Science}, \orgname{Universität Bern}, \orgaddress{\street{Alpeneggstrasse 22}, \city{Bern}, \postcode{3012}, \country{Switzerland}}}

\date{\today}

\abstract{  Let $\Xi_n=\{\xi_1,\dots,\xi_n\}$ be a sample of $n$ independent
  points distributed in a regular closed element $K$ of the extended
  convex ring in $\mathbb{R}^d$ according to a probability measure
  $\mu$ on $K$, admitting a density function. We consider random sets
  generated from the intersection of the translations of $K$ by
  elements of $\Xi_n$, as
  \begin{displaymath}
    X_n=\bigcap_{i=1}^n (K-\xi_i).
  \end{displaymath}
  This work aims to show that the scaled closure of the complement of $X_n$ as $n\to\infty$
  converges in distribution to the closure of the complement zero cell of a Poisson hyperplane
  tessellation whose distribution is determined by the curvature
  measure of $K$ and the behaviour of the density of $\mu$ near the
  boundary of $K$.}

 
\maketitle
\section{Introduction}
\label{sec:introduction}

A random closed set in Euclidean space is a random element in the
family of closed sets in $\R^d$ equipped with the Fell topology.  Limit theorems for random sets are mostly derived for
their Minkowski sums and unions, see \cite{TRS}.

Only recently, \cite{MR4400102} proved a limit theorem for random sets
obtained as the intersection of unit Euclidean balls $x+B_1$ whose
centres $x$ form a Poisson point process of growing intensity
$\lambda$ on the same unit ball. It was shown that these random sets
scaled by $\lambda$ converge in distribution as $\lambda\to\infty$ to
the zero cell of a Poisson hyperplane tessellation. It was also shown
that the volumes of these intersection sets converge in distribution.

Many results from \cite{MR4400102} follow from the studies of
a similar model appearing by taking the ball hull of a sample
$\Xi_n=\{\xi_1,\dots,\xi_n\}$ which consists of i.i.d.\ points
uniformly distributed in the unit Euclidean ball.  The ball hull
\begin{displaymath}
  Q_n=\bigcap_{x\in\R^d, \Xi_n\subset x+B_1} (x+B_1)
\end{displaymath}
is the intersection of all (translated) unit balls which contain the
sample. It is easy to see that the set $X_n$ of all $x\in\R^d$ such
that $\Xi_n\subset x+B_1$ satisfies
\begin{displaymath}
  X_n=\{x:\Xi_n\subset x+B_1\}=\bigcap_{i=1}^n (\xi_i+B_1).
\end{displaymath}
This line of research was initiated in \cite{MR3290422}, where the combinatorial structure of the ball polytope $Q_n$ was explored. In
particular, it was shown that the expectation of the number of faces
of $Q_n$ in dimension 2 converges to a nontrivial limit without taking
any normalisation. This was generalized in \cite{MR4130336}, by
replacing the ball with a convex body in $\R^2$, whose boundary needs to satisfy some requirements.

The ball hull model has been extended in \cite{MR4363583}, where the unit ball was replaced by a general convex body in the space of arbitrary
dimension and it was shown that the convergence of distribution of
$nX_n$ entails the convergence of the combinatorial features of $Q_n$, namely, its $f$-vector, in particular, the numbers of vertices and
facets. It is shown that the latter convergence holds in distribution together with all moments.

The standard closed convex hull of a set is defined as the intersection of all images,
under the action of a group of rigid motions, of a half-space containing the given set. In \cite{ARXIV} a generalisation of this concept is proposed, with a focus on the analysis of the newly defined convex hulls of random samples taken from a fixed convex body.

Following the setting of \cite{MR4363583}, we assume that
$\Xi_n=\{\xi_1,\dots,\xi_n\}$ is a set of $n$ i.i.d.\ points
distributed in a set $K\subset\R^d$ according to a probability measure
$\mu$ and define
\begin{equation}
  \label{eq:1}
  X_n:=\bigcap_{i=1}^n (K-\xi_i).
\end{equation}
In contrast to \cite{MR4363583}, it is not assumed that points are
uniformly distributed in $K$ and the convexity assumption on $K$ is
also dropped. Instead, it is assumed that $K$ is a regular closed
set from the extended convex ring, which is the family of locally
finite unions of convex bodies, see \cite{MR2304055}.

The main result states that, after an appropriate multiplicative
scaling, the complement of the closure of the random closed set $X_n$ converges in distribution to the closure of the complement of the
zero cell of a tessellation in $\R^d$, whose distribution is
determined by the curvature measure of $K$ and the behaviour of $\mu$
near the boundary of $K$. Some limit results on the expectation of the volume of properly scaled $X_n$ are also presented. 

The proof is based on two technical results. First, it relies on
the analysis of the asymptotic properties of a family of measures over Minkowski differences between a set $K$ and a scaled version of another set $L$, which is
\begin{displaymath}
  K\ominus \eps L:=\{x:x+\eps L\subset K\}.
\end{displaymath}
In \cite{MR2304055}, the authors obtain such results for the volume
$V_d$. For each gentle set $K$, whose definition can be found in Section \ref{sec:asympotic}, and each compact set $L$,
\begin{displaymath}
  \lim_{\eps\downarrow0}\frac{1}{\eps}V_d(K\setminus K\ominus \eps L)
  =\int_{\Sphere}h(L,u)^+S_{d-1}(K,d u), 
\end{displaymath}
where $h(L,u)$ is the support function of $L$ and $h(L,u)^+$ is its
positive part, and $S_{d-1}$ is the surface area measure of $K$.

In Section~\ref{sec:asympotic}, we extend this result by replacing the volume with a general finite measure $\mu$, which is absolutely continuous in a neighbourhood of the boundary of $K$ and has density function which behaves like a power of order $\alpha$ near the
boundary of $K$. The limit is given by a similar integral which
depends on the support function $h$ and the asymptotic behaviour of
the density function of $\mu$ near the boundary $\partial K$. The function of this power then appears in the scaling factor of $X_n$ as
$n^\gamma$ with $\gamma=(1+\alpha)^{-1}$.  It is natural that only the
behaviour of $\mu$ near $\partial K$ matters for the asymptotic of
$n^{\gamma} X_n$ since $K-\xi_i$ for $\xi_i$ within any positive
distance of $\partial K$ does not contribute to the intersection in
\eqref{eq:1}.

The second main technical tool relies on the fact that convergence of the inclusion functionals of regular closed random sets implies the
convergence in distribution of the closure of their complements.

\section{Asymptotic properties of integrals over Minkowski differences}
\label{sec:asympotic}

By $\sH^{d-1}$ we denote the Hausdorff measure of dimension $(d-1)$,
and by $\sH^d$ or $V_d$ the Lebesgue measure in $\R^d$. We write $dx$
in integrals with respect to the Lebesgue measure. By $B_r(a)$ we
denote the closed ball in $\R^d$ of radius $r$ and centre
$a\in\R^d$. Given a set $A$ in $\R^d$, denote by $\interior(A)$ 
the interior of $A$, by $\cl(A)$ its closure, and by $A^c$ its
complement.

Let $K\subseteq\R^d$ be a closed set. The metric projection
$\xi_K:\R^d\setminus\exo(K)\to K$ is defined by letting $\xi_K(a)$ be
the unique nearest point to $a$ from $K$, where the exoskeleton
$\exo(K)$ is the set of points which do not admit a unique nearest
point in $K$. The set $\exo(K)$ is measurable and $V_d(\exo(K))=0$,
see Section~2 in \cite{MR2304055}. The reduced normal bundle of $K$ is
\begin{displaymath}
  N(K):=\left\{\left(\xi_K(z),\frac{z-\xi_K(z)}{\|z-\xi_K(z)\|}\right)
    :z\not\in K\cup \exo(K)\right\}.
\end{displaymath}
The set
\begin{displaymath}
  \hat{N}(K):=N(\partial K)\setminus N(K)
\end{displaymath} 
is called the inner reduced normal bundle.  The reach function of $K$
is defined as
\begin{displaymath}
  \delta(K,a,u):=\inf\{t\geq 0:a+tu\in\exo(K)\}, \quad (a,u)\in N(K),
\end{displaymath}
assuming that $\delta(K,a,u)=+\infty$ if the set
$\{t\geq0:a+tu\in\exo(K)\}$ is empty.

A set $K\subset\R^d$ is called regular closed if it coincides with the
closure of its interior. Let $\mathcal{R}$ be the family of regular
closed sets.  A convex set $K$ is regular closed if and only if its
interior is not empty.

A closed set $K$ is said to be gentle if
\begin{itemize}
\item[(G1)] $\sH^{d-1}(\{a\in B:(a,u)\in N(\partial K), u\in \Sphere\})<\infty$
  for all bounded Borel sets $B\subset\R^d$,
\item[(G2)] for $\sH^{d-1}$-almost all $a\in\partial K$, there are
  non-degenerate balls $B_i$ and $B_o$ containing $a$ such that
  $B_i\subset K$ and $\interior B_o\subset K^c$, 
\end{itemize}
see Section~2 in \cite{MR2304055}. A gentle set $K$ is not
necessarily regular closed, for example, a singleton is gentle. Each
convex body in $\R^d$ is gentle. 

Let $K$ be a gentle set. For $\sH^{d-1}$-almost all $a\in\partial K$,
there exists a unique $u\in\Sphere$ such that $(a,u)\in N(K)$ and
$(a,-u)\in \hat{N}(K)$. This follows from (G2) since the tangent balls $B_i$ and $B_o$ at $a$ are unique. Let $C_{d-1}(K,\cdot)$ be the image measure of
$\sH^{d-1}$ on $\partial K$ under the map $a\mapsto (a,u)\in N(K)$,
which is defined $\sH^{d-1}$-almost everywhere on $\partial K$ and is
measurable, see Lemma~6.3 from \cite{MR2031455}. The measure
$C_{d-1}(K,\cdot)$ is called the curvature measure of $K$.

We will use the abbreviation for almost all $(a,u)\in N(K)$ instead
of for $C_{d-1}(K,\cdot)$-almost all $(a,u)\in N(K)$. The same
agreement is used for $\hat{N}(K)$ equipped with the measure
$C_{d-1}^*(K,\cdot)$, where $C^*_{d-1}(K,\cdot)$ is the image measure
of $C_{d-1}(K,\cdot)$ under the reflection
$(a,u)\in N(K)\mapsto(a,-u)\in\hat{N}(K)$, this
reflection is well-defined almost everywhere. 

If $K$ is gentle, for almost all $(a,u)\in N(K)$, it is possible to
express explicitly in terms of the reach function the radii of the
inner and outer balls associated with an $a\in\partial K$ as mentioned
in (G2). For almost all $(a,u)\in N(K)$, we denote
$\delta_\pm:=\delta(\partial K,a,\pm u)$, so that
$B_o:=B_{\delta_+}(a+\delta_+u)$ and
$B_i:=B_{\delta_-}(a-\delta_-u)$. Moreover, $\delta_+>0$ and
$\delta_->0$.

For $L\subseteq\R^d$, its support function
is defined as
\begin{displaymath}
  h(L,u):=\sup\{\langle x, u\rangle: x\in L\}, \quad u\in\R^d.
\end{displaymath}
If $h(L,u)<+\infty$, the supporting hyperplane of $L$ with normal
$u\neq 0$ is given by
\begin{displaymath}
  H(L,u):=\{x\in\R^d: \langle u, x\rangle=h(L,u) \}.
\end{displaymath}
We denote with $h(L,u)^+$ the positive part of $h(L,u)$.

For sets $K$ and $L$ in $\R^d$, 
\begin{displaymath}
  K\ominus L:=\{x\in\R^d: L+x\subseteq K\}
\end{displaymath}
is the Minkowski difference of $K$ and $L$, and the set
\begin{displaymath}
  \check{L}:=\{x\in \R^d:-x\in L\}
\end{displaymath}
is the reflection of $L$ with respect to the origin.

We recall that for almost all $a\in\partial K$, if $(a,u)\in N(K)$, we have that $\delta(K,a,u)=\delta(\partial K,a,u)$ since $N(K)\subset N(\partial K)=N(K)\cup \hat{N}(K)$. The same applies to $\hat{N}(K)$.

\begin{lemma}\label{lemma1}
  Let $K$ be a gentle set and let $L$ be a compact set in $\R^d$. Let
  $f$ be a nonnegative integrable function. Then, for almost all
  $(a,u)\in \hat{N}(K)$, there exist functions $t_+$ and $t_-$, which
  satisfy
  \begin{displaymath}
    \lim_{\eps\downarrow 0}\frac{t_\pm(\eps)}{\eps}=h(L,-u)^+,
  \end{displaymath}
  and such that
  \begin{align*}
    \int_0^{t_+(\eps)}f(a+tu)\,dt
    &\leq\int_0^{\delta(\partial K,a,u)}f(a+tu)\one_{K\setminus
      K\ominus\eps L}(a+tu)\,dt\\
    &\leq\int_0^{t_-(\eps)}f(a+tu)\,dt
  \end{align*}
  for all sufficiently small $\eps$.
\end{lemma}
\begin{proof}
  Fix $(a,u)\in \hat{N}(K)$ such that the inner ball $B_i$ and outer
  ball $B_o$ exist as in (G2). Define the half-space
  \begin{displaymath}
    H^-_{-u}(h(L,-u)):=\{x\in\R^d:\langle x, -u\rangle\leq h(L,-u)\}. 
  \end{displaymath}
  There exists an $r>0$ such that $L\subseteq B_r(0)$, and,
  consequently, $\eps L\subseteq B_{\eps r}(0)$.  The set $\eps L$ is not
  only a subset of $B_{\eps r}(0)$, but also of
  \begin{displaymath}
    C(\eps):=\eps H^-_{-u}(h(L,-u))\cap B_{\eps r}(0).
  \end{displaymath}
  Since $L$ is compact, there is a point $l\in\partial L$ such that
  $\langle l,-u\rangle=h(L,-u)$, which is called the
  support point, and $\eps l\in \eps(H(L,-u)\cap B_r(0))$.

  For all $\eps<r^{-1}\min\{\delta_+,\delta_-,1\}$ and for
  $t\in[0,\delta_-)$, we have $(a+tu)\in (K\setminus K\ominus\eps L)$
  if and only if $a+tu+\eps L$ is not a subset of $K$.  Hence, for
  $t\in[0,\delta_-)$,
  \begin{align*}
    \{t\,:\, a+tu+\eps L \not\subseteq K\}
    &\subseteq\{t \,:\,a+tu+\eps L\not\subseteq B_i\}\\
    & \subseteq\{t \,:\,a+tu+C(\eps)\not\subseteq B_i\}\\
    &=\{t\,:\,a+tu+\eps(H(L,-u)\cap B_r(0))\not\subseteq B_i\},
  \end{align*}
  and 
  \begin{align*}
    \{t \,:\, a+tu+\eps L\not\subseteq K\}
    & =\{t \,:\,(a+tu+\eps L)\cap K^c\not=\emptyset\}\\
    &\supseteq \{t \,:\,(a+tu+\eps L)\cap \interior B_o\not=\emptyset\}\\
    &\supseteq \{t \,:\,(a+tu+\eps l)\in \interior B_o\}\\
    &\supseteq\{t \,:\,a+tu+\eps(H(L,-u)
      \cap B_r(0))\subseteq \interior B_o\}.
  \end{align*}
  We then define 
  \begin{align*}
    t_-(\eps)
    &:=0\vee\inf\{t\in(-\delta_+,\delta_-):
    a+tu+\eps(H(L,-u)\cap B_r(0))\subseteq B_i\},\\
    t_+(\eps)
    &:=0\vee\inf\{t\in(-\delta_+,\delta_-):
      \,a+tu+\eps(H(L,-u)\cap B_r(0))\not\subseteq \interior B_o\},
  \end{align*}
  where $\vee$ stands for supremum.
  
  The value of $t_-(\eps)$ is the distance between $a$ and a point of
  the segment $[a, a+\delta_-u]$. Since the set
  $a+tu+\eps(H(L,-u)\cap B_r(0))$ is invariant under any
  rotation which keeps $u$ unchanged, $t_-(\eps)$ can be calculated by sectioning
  this set with any 2-dimensional plane parallel to $u$ which
  contains $a$. For each $t\in(0,\delta_-)$, the section of $B_{\delta_-}(a+tu)$ is a circle
  $C$. Then $t_-(\eps)$ is the positive part of the sum of
  $\eps h(L,-u)$ and the distance between a chord of the
  circle $C$ of length $2(r^2\eps^2-\eps^2h(L,-u)^2)^{1/2}$
  and the point on the boundary $a$.
  Hence,
  \begin{displaymath}
    t_-(\eps)=(\eps
    h(L,-u)+\delta_- -(\delta_-^2-(r^2\eps^2-\eps^2h(L,-u)^2))^{1/2})^+,
  \end{displaymath}
  where $t_-(\eps)\eps^{-1}\to h(L,-u)^+$ as
  $\eps\downarrow0$. 

  With the same idea used to calculate $t_-(\eps)$, we can see that
  \begin{displaymath}
    t_+(\eps)=(\eps
    h(L,-u)-(\delta_+-(\delta_+^2-(r^2\eps^2-\eps^2h(L,-u)^2))^{1/2}))^+
  \end{displaymath}
  and $t_+(\eps)\eps^{-1}\to h(L,-u)^+$ as $\eps\downarrow0$.
  
  Clearly $0\leq t_+(\eps)\leq t_-(\eps)<\delta^-$ for
  $\eps<r^{-1}\min\{\delta_+,\delta_-,1\}$. The points of
  discontinuity of $\one_{K\setminus K\ominus\eps L}(a+tu)$ happen for
  $t\in[t_+(\eps),t_-(\eps)]$. The identity function
  $\one_{K\setminus K\ominus\eps L}(a+tu)$ is $1$ for all
  $t\in[0,t_+(\eps))$ and is $0$ for all $t\geq t_-(\eps)$.
  
  The proof finishes by observing that $f$ is integrable and positive
  on $[0,\delta^-)$.
\end{proof}

We impose the following conditions on a finite measure $\mu$
supported by $K$.

\begin{enumerate}
\item[(M1)] The measure $\mu$ has compact support and it is absolutely continuous with density $f$.
\item[(M2)] There exists an $\alpha>-1$ such that, for almost all
  $(a,u)\in\hat{N}(K)$, 
  \begin{displaymath}
    \lim_{t\downarrow0}\frac{f(a+tu)}{t^\alpha}=\hat{g}(a,u)=:g(a)\in[0,+\infty),
  \end{displaymath}
  where the function $\hat{g}$ is strictly positive on a subset of
  $\partial K$ of positive measure and is bounded almost everywhere.
\end{enumerate}
  The argument of the limit function $\hat{g}$ is the vector $(a,u)\in \hat{N}(K)$. Since $K$ is gentle, for $\sH^{d-1}$-almost all $a\in\partial K$ there is a unique $u\in\Sphere$ such that $(a,u)\in \hat{N}(K)$, then $g$ is well-defined.

  As a pointwise limit of measurable
  functions, $g$ is also measurable. From its definition and since the support of
  $f$ on $K$ is compact, $g$ has compact support.

\begin{proposition}\label{prop1}
  Let $K$ be a gentle set and let $L$ be a compact set in $\R^d$. Let $\mu$ be a finite measure on $K$ which has compact support, is absolutely continuous in a neighbourhood of $\partial K$ with density $f$ and satisfies (M2). Then
  \begin{displaymath}
    \lim_{\eps\downarrow0}\frac{\mu(K\setminus K\ominus\eps^\gamma
      L)}{\eps}
    =\int_{N(K)}g(a)\frac{(h(L,u)^+)^{\alpha+1}}{\alpha+1} C_{d-1}(K,d(a,u)),
  \end{displaymath}
  where $\gamma=(\alpha+1)^{-1}$.
\end{proposition}
\begin{proof}
  There exists a constant $s>0$ such that the finite measure $\mu$ is
  absolutely continuous on $\partial K+B_s(0)$ since it is a
  neighbourhood of $\partial K$. Moreover, there exists an $r>0$ such
  that $L\subseteq B_r(0)$. 
  
  For $\eps<(s/r)^{1/\gamma}$, Proposition~4 from
  \cite{MR2304055} yields that 
  \begin{align}
    & \mu(K\setminus K\ominus\eps^\gamma L) \notag \\
    & =  \int_{\R^d}f(x) \one_{K\setminus K\ominus\eps^\gamma L}(x)\,dx \label{eqsum}\\
    & = \sum_{i=1}^d i\kappa_i \underset{N(\partial K)}{\int}\overset{\delta(\partial K,a,u)}{\underset{0}{\int}}t^{i-1}f(a+tu) \one_{K\setminus K\ominus\eps^\gamma L}(a+tu)\,dt\,\nu_{d-i}(\partial K,d(a,u)),\notag
  \end{align}
  where the signed measures
  $\nu_0(\partial K,\cdot),\dots,\nu_{d-1}(\partial K,\cdot)$ are
  called support measures of $\partial K$ and the constant $\kappa_i$ is the volume of the $i$-dimensional unit ball, for $i=1,\dots,d$. The signed measures $\nu_0(\partial K,\cdot),\dots,\nu_{d-1}(\partial K,\cdot)$ have locally finite total variation from Corollary 2.5 and (2.13) in \cite{MR2031455}.
  
  The hypotheses of Proposition~4 in \cite{MR2304055} require that $f$
  is bounded. The statement also holds for unbounded
  integrable functions. The sequence of functions
  $f_n:=f\one_{\{f\leq n\}}$ is monotone and each $f_n$ is bounded. By
  the monotone convergence, the identity (\ref{eqsum}) is also
  satisfied by unbounded integrable functions.
  
  All summands of (\ref{eqsum}) with $i\geq2$ are of order
  $o(\eps)$ as $\eps\downarrow0$. The support of
  $\one_{K\setminus K\ominus\eps^\gamma L}$ is contained in
  $\partial K+B_{2r\eps^\gamma }(0)$, so that
  \begin{align*}
    & \Big\lvert \frac{1}{\eps}\int_{N(\partial K)}\int^{\delta(\partial K,a,u)}_0 t^{i-1} f(a+tu)\one_{K\setminus K\ominus\eps^\gamma L}(a+tu)\,dt\,\nu_{d-i}(\partial K,d(a,u))\Big\rvert \\
    & \leq  \frac{1}{\eps}\int^{2r\eps^\gamma}_0\int_{N(\partial K)} t^{i-1+\alpha} \frac{f(a+tu)}{t^\alpha}\,dt\,\lvert\nu_{d-i}\rvert(\partial K,d(a,u))\\
    & \leq \frac{(2r\eps^\gamma)^{i+\alpha}}{(i+\alpha)\eps}\underset{((a,u),t)\in N(\partial K)\times[0,2r\eps^\gamma)}{\esssup}\frac{f(a+tu)}{t^\alpha}\lvert\nu_{d-i}\rvert(\partial K,\partial K(\eps)\times\Sphere)\to0,
  \end{align*}
  as $\eps\downarrow0$, for $i\not=1$, $\gamma=(\alpha+1)^{-1}$, $\alpha>-1$, and where 
  \begin{displaymath}
      \partial K(\eps):=\{x\in\partial K:x+y\in\supp f \text{ for }y\in B_{2r\eps^\gamma}(0)\},
  \end{displaymath}
  which is compact for all sufficiently small $\eps$ since $\supp f$ is compact on $\partial K$. It follows that $\lvert\nu_{d-i}\rvert(\partial K,\partial K(\eps)\times\Sphere)$ is finite for each $i=1,\dots,d-1$, since $\partial K(\eps)$ is bounded and the measures have locally finite total variation. Furthermore, it follows from the assumption (M2) that $f(a+tu)/t^\alpha$ is bounded for almost all $((a,u),t)\in N(\partial K)\times[0,2r\eps^\gamma)$ for $\eps$ small enough.
  
  Proposition~4.1 and Proposition~5.1 from \cite{MR2031455} yield that
  \begin{displaymath}
    2\nu_{d-1}(\partial K,\cdot)=C_{d-1}(K,\cdot)+C_{d-1}^*(K,\cdot).
  \end{displaymath}
  For $\kappa_1=2$,
  \begin{align*}
    & \lim_{\eps\downarrow0}\frac{\mu(K\setminus K\ominus\eps^\gamma L)}{\eps} \\
    & = \lim_{\eps\downarrow0}\frac{2}{\eps}\int_{N(\partial K)}\int^{\delta(\partial K,a,u)}_0 f(a+tu) \one_{K\setminus K\ominus\eps^\gamma L}(a+tu)\,dt\,\nu_{d-1}(\partial K,d(a,u))\\
    & = \lim_{\eps\downarrow0}\frac{2}{2\eps}\int_{N(K)}\int^{\delta_+}_0 f(a+tu) \one_{K\setminus K\ominus\eps^\gamma L}(a+tu)\,dt\,C_{d-1}(K,d(a,u))\\
    & +\frac{2}{2\eps}\int_{\hat{N}(K)}\int^{\delta_-}_0 f(a+tu) \one_{K\setminus K\ominus\eps^\gamma L}(a+tu)\,dt\,C_{d-1}^*(K,d(a,u))\\
    & = \lim_{\eps\downarrow0}\frac{1}{\eps}\int_{\hat{N}(K)}\int^{\delta_-}_0 f(a+tu) \one_{K\setminus K\ominus\eps^\gamma L}(a+tu)\,dt\,C_{d-1}^*(K,d(a,u)),
  \end{align*}
  since the density function $f$ vanishes outside $K$.
  Following Lemma~\ref{lemma1}, for almost all $(a,u)\in\hat{N}(K)$,
  \begin{align*}
    g(a)\frac{(h(L,-u)^+)^{\alpha+1}}{\alpha+1} & = \lim_{\eps\downarrow0} \inf_{s\in[0,t_+(\eps^\gamma))}\frac{f(a+su)}{s^\alpha}\frac{1}{\eps}\int_0^{t_+(\eps^\gamma)}t^\alpha\,dt\\
    &\leq \lim_{\eps\downarrow0}\frac{1}{\eps}\int_0^{t_+(\eps^\gamma)}f(a+tu)\,dt\\
    & \leq \lim_{\eps\downarrow0}\frac{1}{\eps}\int^{\delta_-}_0f(a+tu)\one_{K\setminus K\ominus\eps^\gamma L}(a+tu)\,dt\\
    &\leq\lim_{\eps\downarrow0}\frac{1}{\eps}\int_0^{t_-(\eps^\gamma)}f(a+tu)\,dt,\\
    & \leq \lim_{\eps\downarrow0} \sup_{s\in[0,t_- (\eps^\gamma))}\frac{f(a+su)}{s^\alpha}\frac{1}{\eps}\int_0^{t_-(\eps^\gamma)}t^\alpha\,dt\\
    & = g(a)\frac{(h(L,-u)^+)^{\alpha+1}}{\alpha+1}.
  \end{align*}

  Thus, 
  \begin{multline*}
    F_\eps(a,u):=\frac{1}{\eps}\int^{\delta_-}_0f(a+tu)\one_{K\setminus
      K\ominus\eps^\gamma L}(a+tu)\,dt\\
    \to g(a)\frac{(h(L,-u)^+)^{\alpha+1}}{\alpha+1}
    \quad \text{as}\; \eps\downarrow 0
  \end{multline*}
  for almost all $(a,u)\in\hat{N}(K)$. The limiting function is
  bounded and has compact support, because $g$ is a bounded function with compact support and $L$ is compact.
  
  The sequence $F_\eps$ is bounded for all $\eps>0$ and for almost all
  $(a,u)\in\hat{N}(K)\cap((\supp f\cap \partial K)\times \Sphere)$,
  which is compact. We can then find an upper bound and apply the
  dominated convergence theorem.

  We conclude by noticing that
  \begin{multline*}
       \int_{\hat{N(K)}}g(a)\frac{(h(L,-u)^+)^{\alpha+1}}{\alpha+1} C^*_{d-1}(K,d(a,u))=\\
       \int_{N(K)}g(a)\frac{(h(L,u)^+)^{\alpha+1}}{\alpha+1}
       C_{d-1}(K,d(a,u)). 
  \end{multline*}
\end{proof}

\begin{remark}
  For $d\geq2$, the result does not hold for $\alpha\leq-1$. The reason is
  due to the summands of (\ref{eqsum}). There is no
  $\gamma\in\R$ such that the summand with index $i=1$ converges and
  others do not converge to 0. For $d=1$ there is only one summand, so
  this issue does not emerge.
\end{remark}

\begin{example}
  Let $f(x)=\|\xi_{\partial K}(x)-x\|^\alpha$ with
  $\alpha\in(-1,\infty)$ and $x\in K$, where $K$ is a gentle compact set, and $f=0$
  outside $K$. Then
  \begin{displaymath}
    \lim_{t\downarrow0}\frac{f(a+tu)}{t^\alpha}=\|u\|^\alpha=1
  \end{displaymath}
  for almost all $(a,u)\in\hat{N}(K)$. In this case $g(a)=1$ for all $a\in\partial K$ and
  \begin{displaymath}
    \lim_{\eps\downarrow0}\frac{1}{\eps} \int_{\R^d}f(x) \one_{K\setminus K\ominus\eps^\gamma L}(x)\,d\sH^d(x) = \int_{N(K)}\frac{(h(L,u)^+)^{\alpha+1}}{\alpha+1} C_{d-1}(K,d(a,u)).
  \end{displaymath}
\end{example}

\section{Limit theorem for intersections of translated elements of the
  extended convex ring}
\label{sec:thm}

A random closed set $X$ in $\R^d$ is a measurable map from a
probability space to the space $\sF$ of closed sets in $\R^d$ endowed
with the Borel $\sigma$-algebra generated by the Fell topology, see
\cite{TRS}. The base of the Fell topology consists of finite
intersections of the sets $\{F\in\sF: F\cap G\not=\emptyset\}$ and
$\{F\in\sF: F\cap L=\emptyset\}$ for all open $G$ and compact $L$ in
$\R^d$.

It is known that the distribution of a random closed set is uniquely
determined by its capacity functional defined as
\begin{displaymath}
  T_X(L):=\Prob{X\cap L\not=\emptyset},
\end{displaymath}
where $L$ runs through the family of compact sets in $\R^d$. A
sequence of random closed sets $(X_n)_{n\geq 1}$ in $\R^d$ converges
in distribution to a random closed set $X$ (notation $X_n\dto X$) if
the corresponding probability measures on $\sF$ weakly converge. This
is the case if and only if 
\begin{displaymath}
   T_{X_n}(L)\to T_X(L)\quad \text{as}\; n\to\infty
\end{displaymath}
for each compact set $L$ such that $T_X(L)=T_X(\interior(L))$, see
Theorem~1.7.7 from \cite{TRS}.

We recall that a convex body is a non-empty compact convex set.  The
extended convex ring $\sU$ is a family of closed sets which are
locally finite unions of convex bodies of $\R^d$, meaning that each
compact set intersects at most a finite number of convex bodies that
generate an element of $\sU$. We assume that the empty set belongs
to $\sU$.  Clearly, the family $\sU$ is closed under intersections.

Let $K$ be a non-empty regular closed element of $\sU$, and let $\mu$ be
a probability measure on $K$ satisfying (M1) and (M2). Consider the set
\begin{displaymath}
  \Xi_n:=\{\xi_1,\dots,\xi_n\}
\end{displaymath}
composed of $n$ independent points in $K$ distributed according to
$\mu$.  Let $X_n$ be a random closed set defined as follows
\begin{equation}\label{def:X_n}
    X_n:=\bigcap_{i=1}^n(K-\xi_i).
\end{equation}

Let $\sP_K:=\{(t_i,u_i):i\geq 1\}$ be a Poisson point process on
$(0,\infty)\times\Sphere$ with intensity measure $\nu$, which is the
product of an absolutely continuous measure on $(0,\infty)$ with
density $t^\alpha$ and a measure $\hat{\nu}$ on $\Sphere$ defined as
\begin{equation}\label{support}
    \hat{\nu}(D):=\int_{N(K)}\one_{\{u\in D\}}g(a)C_{d-1}(K,d(a,u)),
\end{equation}
for $D\subset\Sphere$, where $g(a)$ is given by the property (M2) of
$\mu$.

The point process $\sP_K$ corresponds to the family of
hyperplanes $\{x\in\R^d:\langle x,u_i\rangle=t_i\}$, $i\geq1$, which
splits the space into disjoint cells and is said to be a Poisson
hyperplane tessellation of $\R^d$, see \cite{SIG}. The zero cell of
this tessellation is the random convex set
\begin{equation}\label{eq:Z}
  Z=\bigcap_{i\geq1} \{x\in\R^d:\langle x,u_i\rangle\leq t_i\}.
\end{equation}
We now formulate our main result.

\begin{theorem}
  \label{thm:sets-convergence}
  Assume that $\mu$ satisfies (M1) and (M2) on $K$ for $\alpha>-1$. If $K$ is a regular closed element of $\sU$, then
  \begin{displaymath}
    n^{\gamma} \cl(X_n^c) \dto \cl(Z^c) \quad \text{as}\; n\to\infty,
  \end{displaymath}
  where $\gamma=(\alpha+1)^{-1}$. Furthermore, if $K$ is a convex set with non-empty interior, then
  \begin{displaymath}
    n^{\gamma} X_n \dto Z \quad \text{as}\; n\to\infty.
  \end{displaymath}
\end{theorem}
\begin{definition}
    A random closed set $X$ in $\R^d$ is said to be regular closed if $X$ almost surely belongs to the family $\mathcal{R}$ of regular closed sets. 
\end{definition}
For further details about regular closed random sets, see Section~1.1.7 of \cite{TRS}.

The following result is a variant of Theorem~7.5 from 
\cite{ARXIV} with an identical proof. 

\begin{lemma}\label{lemma:conv-reg-clo}
  Let $Y$ and $Y_n$, $n\in\N$ be regular closed random sets in $\R^d$. If
  \begin{displaymath}
    \Prob{L\subseteq Y_n}\to\Prob{L\subseteq Y} \quad \text{as }n\to\infty
  \end{displaymath}
  for all regular closed compact sets $L$ such that
  $\Prob{L\subseteq Y}=\Prob{L\subseteq \interior Y}$, then $\cl(Y_n^c)\dto \cl(Y)$
  as $n\to\infty$.
\end{lemma}
\begin{proof}
    The family of regular closed sets is a separating class,  see Definition 1.1.48 in \cite{TRS}. It follows from Corollary 1.7.14 in \cite{TRS} that for the convergence in distribution it suffices to check that
    \begin{displaymath}
        \Prob{\cl(Y_n^c)\cap L\not=\emptyset}\to\Prob{\cl(Y^c)\cap L\not=\emptyset} \quad \text{as }n\to\infty,
    \end{displaymath}
    for all regular closed compact $L$, which are continuity sets for $\cl(Y^c)$. The latter means that
    \begin{displaymath}
        \Prob{\cl(Y^c)\cap L=\emptyset}=\Prob{\cl(Y^c)\cap \interior(L)=\emptyset}.
    \end{displaymath}
    Fix a regular closed compact set $L$ which is a continuity set. Since
    \begin{displaymath}
        \Prob{\cl(Y^c)\cap L=\emptyset}=\Prob{L\subseteq \interior(Y)}
    \end{displaymath}
    and
    \begin{displaymath}
        \Prob{\cl(Y^c)\cap \interior(L)=\emptyset}=\Prob{\interior(L)\subseteq \interior(Y)},
    \end{displaymath}
    we conclude that 
    \begin{displaymath}
        \Prob{L\subseteq Y}\leq\Prob{\interior(L)\subseteq \interior(Y)}=\Prob{L\subseteq \interior(Y)}\leq\Prob{L\subseteq Y},
    \end{displaymath}
    so that $\Prob{L\subseteq Y}=\Prob{L\subseteq \interior(Y)}$.

    Let $\eps_k$ be a sequence of positive numbers such that $\eps_k\downarrow0$ as $k\to\infty$, and
    \begin{displaymath}
        \Prob{L+B_{\eps_k}\subseteq Y}=\Prob{L+B_{\eps_k}\subseteq \interior(Y)}.
    \end{displaymath}
    Sending $n\to\infty$ in the chain of inequalities 
    \begin{displaymath}
        \Prob{L+B_{\eps_k}\subseteq Y_n}\leq\Prob{L\subseteq \interior(Y_n)}=\Prob{\cl(Y_n^c)\cap L=\emptyset}\leq\Prob{L\subseteq Y_n}.
    \end{displaymath}
    Then, following $\Prob{L\subseteq Y_n}\to\Prob{L\subseteq Y}$, we conclude that
    \begin{align*}
        \Prob{L+B_{\eps_k}\subseteq Y}& \leq\liminf_{n\to\infty}\Prob{\cl(Y_n^c)\cap L=\emptyset}\\
        & \leq \limsup_{n\to\infty}\Prob{\cl(Y_n^c)\cap L=\emptyset}\leq \Prob{L\subseteq Y}.
    \end{align*}
    Finally, note that
    \begin{displaymath}
        \Prob{L+B_{\eps_k}\subseteq Y}\uparrow\Prob{L\subseteq \interior(Y)}=\Prob{L\subseteq Y} \quad \text{as }k\to\infty.
    \end{displaymath}
\end{proof}

\begin{remark}\label{rem:gentle}
  By Proposition 2 in \cite{MR2304055}, each regular closed set in the
  convex ring is gentle. Since properties (G1) and (G2) defining a
  gentle set are local, each regular closed element of the extended
  convex ring is also gentle.
\end{remark} 

In general, the family $\sR$ of regular closed sets is not closed
under the intersection. The following lemmas show that $X_n$ is a regular closed random set for each $n\in\N$.

\begin{lemma}\label{lemma:intersection}
  Let $A$ and $B$ be regular closed elements of $\sU$. Then
  \begin{displaymath}
    C:=\{x\in\R^d: A\cap(B-x)\not\in \sR\}
  \end{displaymath}
  is measurable, and its Lebesgue measure is zero.
\end{lemma}
\begin{proof}
  Assume that $A:=\cup_{i=1}^\infty L_i$ and
  $B:=\cup_{j=1}^\infty K_j$, where $L_i$ and $K_j$ are  convex bodies 
  in $\R^d$ for each $i,j\in\N$. Since $A$ and $B$ are regular closed,
  we can also assume without loss of generality that $L_i$ and $K_j$
  have not-empty interior.  Define
  \begin{align*}
    C_{ij}
    & :=\{x\in\R^d:L_i\cap(K_j-x)\not=\cl\big(\interior(L_i\cap(K_j-x))\big)\} \\
    & =\{x\in\R^d:L_i\cap(K_j-x)\not=\emptyset, \interior(L_i)\cap\interior(K_j-x)=\emptyset\},
  \end{align*}
  where the second equality follows from the fact that $L_i$ and $K_j$
  are convex bodies.

  The set of possible translations of $K_j$, that intersect $L_i$, is
  the Minkowski sum $\check{K_j}+L_i$, which is a
  convex body itself. Since we consider the translations by taking the
  opposite of a point in $\R^d$, we have that
  \begin{align*}
    \check{C}_{ij}
    & =(\check{K_j}+L_i)\setminus(\interior(\check{K_j})+\interior(L_i)) \\
    & = (\check{K_j}+L_i)\setminus\interior(\check{K_j}+L_i) \\
    & = \partial (\check{K_j}+L_i).
  \end{align*}
  In general,
  $\interior(\check{K_j})+\interior(L_i)\subseteq\interior(\check{K_j}+L_i)$,
  but, in this case, we have equality because the two sets are convex
  bodies. Therefore, $V_d(C_{ij})=0$, since $C_{ij}$ is the boundary
  of a convex body.

  We now show that $C\subseteq\cup_{i,j=1}^\infty C_{ij}$. If this is
  the case, then $C$ is measurable since the $\sigma$-algebra is
  complete and any subset of a measurable set of null measure is
  measurable, and
  \begin{displaymath}
    V_d(C)\leq V_d(\cup_{i,j=1}^\infty C_{ij})\leq \sum_{i,j=1}^\infty V_d(C_{ij})=0.
  \end{displaymath}
  We have
  \begin{align*}
    C
    & =\{x\in\R^d: \cup_{i=1}^\infty L_i\cap(\cup_{j=1}^\infty K_j-x)\not=\cl\big(\interior(\cup_{i=1}^\infty L_i\cap\cup_{j=1}^\infty (K_j-x))\big)\}\\
    & = \{x\in\R^d: \cup_{i,j=1}^\infty (L_i\cap( K_j-x))\supsetneq\cl\big(\interior(\cup_{i=1}^\infty L_i\cap\cup_{j=1}^\infty (K_j-x))\big)\}.
  \end{align*}
  
  We recall that, given two countable families $(A_i)_{i\geq1}$ and
  $(B_i)_{i\geq1}$ of subsets of $\R^d$, such that $A_i\supseteq B_i$
  for each $i\in\N$ and
  $\cup_{i=1}^\infty A_i\supsetneq \cup_{i=1}^\infty B_i$, then there
  exist at least one $i\in\N$ such that $A_i\supsetneq B_i$. If
  $x\in C$, then
  \begin{align*}
    \cup_{i,j=1}^\infty (L_i\cap( K_j-x))
    & \supsetneq\cl\big(\interior(\cup_{i=1}^\infty L_i)\cap\interior(\cup_{j=1}^\infty (K_j-x))\big)\\
    & \supseteq\cup_{i,j=1}^\infty \cl\big(\interior(L_i\cap(K_j-x))\big),
  \end{align*}
  so there are $i,j\in\N$ such that $x\in C_{ij}$. Then
  $C\subseteq\cup_{i,j=1}^\infty C_{ij}$.
\end{proof}

\begin{lemma}\label{lemma:Xnregclo}
  Let $\mu$ be an absolutely continuous measure on a regular closed
  set $K$, which is an element of $\sU$. Then $X_n$ is a regular closed random set for each $n\in\N$.
\end{lemma}
\begin{proof}
  The proof relies on the induction. The step $n=1$ is clear.
  Assume that $X_{n-1}$ is a.s.\ regular closed. Define
  \begin{displaymath}
    A_n:=\{(x_1,\dots,x_n)\in(\R^d)^n: \cap_{i=1}^n(K-x_i)\in\sR\},
  \end{displaymath}
  for $n\in\N$. Then $\Prob{(\xi_1,\dots,\xi_{n-1})\in A_{n-1}}=1$. We
  recall the notation $\Xi_n:=\{\xi_1,\dots,\xi_n\}$. Since $K$ and
  $X_{n-1}$ are almost surely regular closed elements of $\sU$,
  Lemma~\ref{lemma:intersection} yields that
  \begin{align*}
    & \Prob{X_n=\cl(\interior(X_n))}
      =\E\Big(\Prob{X_n=\cl(\interior(X_n))|\Xi_{n-1}}\Big)  \\
    & = \E\Big(\one_{\Xi_{n-1}\in A_{n-1}}
      \Prob{X_n=\cl(\interior(X_n))|\Xi_{n-1}}\Big)\\
    & = \E\Big(\one_{\Xi_{n-1}\in A_{n-1}}
      \Prob{\{\xi: X_{n-1}\cap (K-\xi)=\cl(\interior(X_{n-1}\cap(K-\xi)))\}|\Xi_{n-1}}\Big) \\
    & = 1. \qedhere
  \end{align*}
\end{proof}

\begin{proof}[Proof of Theorem~\ref{thm:sets-convergence}]
  Assume that $K$ is a regular closed element of $\sU$. Let $L$ be a compact set in $\R^d$. Then
  \begin{align*}
    \Prob{L\subseteq n^{\gamma}X_n}
    & = \Prob{n^{-\gamma}L\subseteq (K-\xi_i) \text{ for all }i=1,\dots,n}\\
    & = \big(1-\Prob{n^{-\gamma}L\not\subseteq(K-\xi)}\big)^n\\
    & = \big(1-\Prob{\xi+n^{-\gamma}L\not\subseteq K}\big)^n\\
    & = \big(1-\Prob{\xi\not\in K\ominus n^{-\gamma}L}\big)^n.
  \end{align*}
  
  Since $K$ is a regular closed element of $\sU$, it is gentle, see
  Remark~\ref{rem:gentle}, and $L$ is compact, Proposition~\ref{prop1}
  yields that
  \begin{align*}
    \lim_{n\to\infty} n\Prob{\xi\not\in K\ominus n^{-\gamma}L}
    & = \lim_{n\to\infty} n\mu(K\setminus K\ominus n^{-\gamma}L)\\
    & =\int_{N(K)}g(a)\frac{(h(L,u)^+)^{\alpha+1}}{\alpha+1} 
      C_{d-1}(K,d(a,u))<\infty.
  \end{align*}
  Hence,
  \begin{displaymath}
    \lim_{n\to\infty}\Prob{L\subseteq n^{\gamma}X_n}
    =\exp\left(-\int_{N(K)}g(a)
      \frac{(h(L,u)^+)^{\alpha+1}}{\alpha+1} C_{d-1}(K,d(a,u))\right).
  \end{displaymath}

  Let $Z$ be the zero cell of the tessellation generated by the point
  process $\sP_K$. The random convex set $Z$ satisfies
  \begin{align*}
    \Prob{L\subseteq Z}
    & = \Prob{h(L,u)^+\leq t \text{ for all }(t,u)\in\sP_K}\\
    & = \exp(-\nu(\{(t,u)\in (0,\infty)\times \Sphere: h(L,u)^+>t\}))\\
    & =
      \exp\left(-\int_{N(K)}g(a)\frac{(h(L,u)^+)^{\alpha+1}}{\alpha+1} 
      C_{d-1}(K,d(a,u))\right).
  \end{align*}
  
  It follows that,
  \begin{displaymath}
    \lim_{n\to\infty}\Prob{L\subseteq n^{\gamma}X_n}=\Prob{L\subseteq Z}.
  \end{displaymath}
  
  The zero cell $Z$ of the tessellation is a regular closed random
  set. Indeed, $Z$ is a convex closed subset of $\R^d$. Moreover, all
  half-spaces from \eqref{eq:Z} contain the origin in their interior
  almost surely, so that the interior of $Z$ is not empty with
  probability one.

  The convergence of $n^\gamma\cl(X_n^c)$ in distribution to $\cl(Z^c)$ follows by Lemma~\ref{lemma:conv-reg-clo}, taking into
  account that $n^\gamma X_n$ is a regular closed random set by
  Lemma~\ref{lemma:Xnregclo} for each $n\in \N$ and $Z$ is a regular
  closed random set.
  
  Let $K$ be a non-empty regular closed convex set. By Lemma 7.4 in \cite{ARXIV} and the continuity theorem, the sequence of random sets $n^{\gamma} X_n$ converges in distribution $Z$ as $n\to\infty$.
\end{proof}

It depends on the support of $\hat{\nu}$ in (\ref{support}), whether the zero cell $Z$ is unbounded or bounded almost surely. The support of $\hat{\nu}$ is contained in a closed hemisphere of $\Sphere$ if and only if $Z$ is unbounded almost surely. Indeed, assume that the support of $\hat{\nu}$ is contained in a closed hemisphere of $\Sphere$, then the dual cone of the cone generated by the support of $\hat{\nu}$ is contained in $Z$ a.s. by construction. Assume that $Z$ is unbounded a.s., it follows from the construction of $Z$ that there is at least a fixed $u\in\Sphere$ such that no hyperplane generated by $\sP_K$ intersects the cone generated by $u$. Then the support of $\hat{\mu}$ must be contained in the dual cone of the cone generated by $u$, which is a closed hemisphere of $\Sphere$.

Assume that $K$ is a convex body whose interior is non-empty, the sequence of random convex bodies $n^\gamma X_n$ generated by $K$ can converge in distribution to a random set which is unbounded almost surely.

\begin{example}
  Let $K$ be the unit ball centred in the origin and let $\mu$ be a
  probability measure whose limit of the density function $g$ has
  support contained in a closed hemisphere so that the support of
  $\hat{\nu}$ is also contained in a closed hemisphere. Then the
  sequence of random sets $n^\gamma X_n$ converges in distribution to
  a random set which is unbounded almost surely.
\end{example}

It is straightforward to deduce the convergence of all power moments
of the volume restricted to a compact set.

\begin{proposition}
  Let $M$ be a compact set and let $K$ be a gentle set. Let $n^\gamma X_n$ and $Z$ be defined as in $(\ref{def:X_n})$ and $(\ref{eq:Z})$ respectively. Then, for every $m\in\N$,
  \begin{displaymath}
     \E V_d(n^\gamma X_n\cap M)^m \to \E V_d(Z\cap M)^m \quad \text{ as }n\to\infty.
  \end{displaymath}
\end{proposition}
\begin{proof}
  Following the same steps of proof of Theorem \ref{thm:sets-convergence} and by Proposition \ref{prop1}, we notice that, for $x_1,\dots,x_m\in\R^d$,
  \begin{multline*}
      \Prob{x_1,\dots,x_m\in n^\gamma X_n}=\Prob{\{x_1,\dots,x_m\}\subseteq n^\gamma X_n}\to\\ \Prob{\{x_1,\dots,x_m\}\subseteq Z}=\Prob{x_1,\dots,x_m\in Z},
  \end{multline*}
  as $n\to\infty$, even if $K$ is a gentle set.
  
  Then, by the dominated convergence theorem,
  \begin{align*}
    \E V_d(n^\gamma X_n\cap M)^m
    & =\int_{M}\dots\int_{M} \Prob{x_1,\dots,x_m\in n^\gamma X_n}\,dx_1\dots dx_m\\
    & \to \int_M\dots\int_{M} \Prob{x_1,\dots,x_m\in Z}\,dx_1\dots dx_m\quad \text{ as }n\to\infty\\
    & = \E V_d(Z\cap M)^m,
  \end{align*}
  since $M$ is bounded.
\end{proof}

In general, it does not hold that if $n^\gamma X_n$ and $Z$ are bounded almost surely for every $n\in\N$, then the sequence $\E V_d(n^\gamma X_n)$ converges to $\E V_d(Z)$.

If $K$ is convex, then $X_n$ is also convex. In this case, it is possible to consider its intrinsic volumes $V_j$, for $j=0,\dots,d$, which are defined by the Steiner formula, see Theorem 3.10 in \cite{SIG}.

\begin{corollary}
  Let $K$ be a convex body with a non-empty interior. Let $\mu$ be a probability measure which satisfies assumptions (M1) and (M2). Let $n^\gamma X_n$ and $Z$ be defined as in $(\ref{def:X_n})$ and $(\ref{eq:Z})$ respectively. Furthermore, assume that the support of $\hat{\nu}$, as defined in $(\ref{support})$, is not contained in a hemisphere of $\Sphere$. Then, for $j=0,\dots,d$,
    \begin{displaymath}
      V_j(n^\gamma X_n)=n^{\gamma j} V_j(X_n)\dto V_j(Z)\quad \text{as }n\to\infty.
    \end{displaymath}
\end{corollary}
\begin{proof}
    The intrinsic volumes $V_j$ are continuous with respect to the convergence in the Fell topology restricted to the family of convex bodies, see Remark 3.22 in \cite{SIG}. The random closed set $X_n$ is almost surely a convex body with non-empty interior since $K$ is a convex body with non-empty interior. Since the support of $\hat{\mu}$ is not contained in a hemisphere of $\Sphere$, $Z$ is almost surely a convex body with non-empty interior. Then the convergence in distribution of $n^\gamma X_n$ to $Z$ is assured from the second part of the statement of Theorem \ref{thm:sets-convergence}. By the continuity theorem, for $j=0,\dots,d$,
    \begin{displaymath}
      V_j(n^\gamma X_n)=n^{\gamma j} V_j(X_n)\dto V_j(Z)\quad \text{as }n\to\infty.
    \end{displaymath}
\end{proof}
\begin{example}
  Let $K$ be the union of the two disjoint balls $B_1(0)$ and
  $B_1(x)$, with $|x|>2$, and let $\mu$ be the uniform distribution on
  $B_1(0)$. The set of sample
  points $\Xi_n=\{\xi_1,\dots,\xi_n\}$ is a subset of $B_1(0)$ almost
  surely. The random set $n X_n$ is the disjoint union of the convex
  random body
  \begin{displaymath}
    n\Tilde{X}_n:=n\big(\cap_{i=1}^n(B_1(0)-\xi_i)\big),
  \end{displaymath}
  and its translate by $nx$. 

  The set $n\Tilde{X}_n$ converges in distribution to the zero cell $Z$ and the set $n\Tilde{X}_n + nx$ converges in distribution to the empty set as $n\to\infty$.
  
  Then
  \begin{displaymath}
    V_d(n X_n)=2V_d\big(n\Tilde{X}_n\big)\dto
    2V_d(Z) \quad  \text{as }n\to\infty.
  \end{displaymath}
  But $n X_n$ also converges in distribution to $Z$. Hence,
  $V_d(n X_n)$ does not converge in distribution to $V_d(Z)$
  as $n\to\infty$.  In particular, from Proposition 5.4 in \cite{MR4363583},
  \begin{displaymath}
    \E V_d(n X_n)=2\E V_d\big(n\Tilde{X}_n\big)\to 2\E
    V_d(Z)\quad  \text{as }n\to\infty,
  \end{displaymath}
  so that $\E V_d(n X_n)$ does not converge to $\E
  V_d(Z)$ as $n\to\infty$.
\end{example}

\section*{Acknowledgements}
\label{sec:acknowledgements}

The author is grateful to Prof. Ilya Molchanov for all his insightful suggestions and for his patience in correcting the various drafts of this work. 

This work was supported by the Swiss Enlargement Contribution in the Croatian--Swiss Research Programme framework. (Project number IZHRZ0\_180549).

The author declares to not have any conflicts of interest. 

Data sharing is not applicable to this article as no data-sets were generated or analyzed
during the current study.

\bibliography{Intersections_of_randomly_translated_sets}

\end{document}